\documentclass[journal]{IEEEtran}
\ifCLASSINFOpdf
\else
\fi
\usepackage{def}

\begin{document}
%
\title{Decentralized Constraint-Coupled Optimization with Inexact Oracle}
%
%
%
%
%
\author{Jingwang Li and Housheng Su
  \thanks{Email: jingwangli@outlook.com, houshengsu@gmail.com.}
}

\maketitle

\begin{abstract}
  We propose an inexact decentralized dual gradient tracking method (iDDGT) for decentralized optimization problems with a globally coupled equality constraint. Unlike existing algorithms that rely on either the exact dual gradient or an inexact one obtained through single-step gradient descent, iDDGT introduces a new approach: utilizing an inexact dual gradient with controllable levels of inexactness. Numerical experiments demonstrate that iDDGT achieves significantly higher computational efficiency compared to state-of-the-art methods. Furthermore, it is proved that iDDGT can achieve linear convergence over directed graphs without imposing any conditions on the constraint matrix. This expands its applicability beyond existing algorithms that require the constraint matrix to have full row rank and undirected graphs for achieving linear convergence.
\end{abstract}


%
\IEEEpeerreviewmaketitle

\section{Introduction}
\label{intro}
Recently decentralized optimization has gained significant popularity in numerous fields due to its promising applications in areas such as large-scale machine learning, distributed control, decentralized estimation, smart grids, and more \cite{yang2019survey,gabrielli2023survey,chen2021distributed,wang2022distributed,zhai2022distributionally,chen2023distributed,gao2021event}.
This work focuses on addressing the decentralized optimization problem
\begin{equation} \label{original_pro} \tag{P1}
  \begin{aligned}
    \min_{x_i \in \mR^{d_i}} \  & \sum_{i=1}^n f_i(x_i)    \\
    \text{s.t.} \               & \sum_{i=1}^{n}A_ix_i = b
  \end{aligned}
\end{equation}
over a directed network consisting of $n$ agents, where $f_i:\mR^{d_i} \rightarrow \mR$ and $A_i \in \mR^{p \times d_i}$ are completely private for agent $i$ and cannot be shared with its neighbors, while $b \in \mR^{p}$ is public for all agents. Without loss of generality, assume that there exist at least one finite solution of \cref{original_pro}. The constraint $\sum_{i=1}^{n}A_ix_i = b$ couples the decision variables of all agents, making it a decentralized constraint-coupled optimization problem \cite{falsone2020tracking,li2022implicit}. Notably, \cref{original_pro} can conver lots of practical optimization problems, such as distributed resource allocation \cite{falsone2020tracking,zhang2020distributed} and decentralized vertical fedrated learning \cite{chang2014multi,li2022nested}.

One can observes that the dual of \cref{original_pro} has the same form with the classical decentralized unconstrained optimizaton (DUO) problem, leading to the natural idea of applying existing DUO algorithms to its dual. This approach has been adopted in numerous previous works \cite{chang2014multi,yi2016initialization,falsone2017dual,alghunaim2019proximal,falsone2020tracking,zhang2020distributed,li2022implicit,alghunaim2021dual,li2022nested,lu2022decentralized,wu2022distributed,su2021distributed,falsone2023augmented}. However, a key challenge lies in dealing with the dual gradient or dual subgradient. A straightforward approach is to use the exact dual gradient and follow the same steps as in DUO algorithms. This involves applying a suitable DUO algorithm to the dual of \cref{original_pro}, resulting in a decentralized algorithm for \cref{original_pro}, and the resulting algorithm is essentially a special case of the original DUO algorithm. The only thing you need to do is to use a gradient based DUO algorithm if the dual funtion is differentiable, and a subgradient based one if the dual function is non-differentiable. Related works include \cite{chang2014multi,falsone2017dual,falsone2020tracking,zhang2020distributed,wu2022distributed,falsone2023augmented}. Despite its convenience, the aforementioned approach has a common drawback: the use of the exact dual gradient necessitates solving a subproblem exactly at each iteration. This can be computationally expensive and even infeasible in practice, particularly when dealing with nonlinear objective functions \cite{devolder2014first}.

A simple and widely adopted solution to address the above limitation is to use an inexact dual gradient instead of the exact one, which has been extensively explored in existing works  \cite{chang2014multi,yi2016initialization,alghunaim2019proximal,li2022implicit,alghunaim2021dual,li2022nested,lu2022decentralized,su2021distributed}. In these works, a common approach to obtaining the inexact dual gradient is to employ single-step (proximal) gradient descent, which leads to an approximate solution of the subproblem. This approach can be viewed as minimizing the first-order approximation of the objective function. By introducing this approximation, the resulting algorithms do not require solving the subproblem exactly at each iteration, making them computationally feasible and easy to implement.
However, there are some concerns regarding the aforementioned approximate method:
\begin{enumerate}
  \item The suboptimality of the approximate solution obtained through single-step gradient descent is uncontrollable \footnote{Certainly, we can control the suboptimality within a certain range by adjusting the step size of single-step gradient descent. However, when we say the suboptimality is "uncontrollable," we mean that we cannot make the suboptimality arbitrarily small, regardless of the step size chosen.}, this implies that we are unable to control the gap between the inexact dual gradient and the exact one, which is crucial for the algorithm's performance. As a result, the ability to control and optimize the overall performance of the algorithm may be significantly limited.
  \item There are multiple methods available for solving the subproblem, such as multi-step gradient descent, Nesterov's accelerated gradient descent (AGD) \cite{nesterov2018lectures}, Newton's method, and others. Relying solely on single-step gradient descent is overly inflexible. Intuitively, incorporating AGD or even second-order methods could potentially enhance the overall performance of the algorithm.
  \item The computation cost and communication cost vary widely across different decentralized optimization scenarios. In some scenarios, computation is cheap while communication is costly, whereas the opposite is true for others. Intuitively, when computation is cheap but communication is expensive, utilizing a more accurate dual gradient (which requires more computation steps to solve the subproblem) may lead to a decrease in the total convergence time. However, this strategy is impractical for algorithms based on single-step gradient descent because we cannot arbitrarily control the accuracy of the inexact dual gradient.
\end{enumerate}
Therefore, we aim to develop a new scheme that can address the aforementioned potential concerns.

Besides, we are also interested in addressing another open problem: Can we design an algorithm that can linearly solve \cref{original_pro} under a less restrictive condition on $A_i$? Currently, the weakest condition obtained in \cite{li2022implicit,alghunaim2021dual,li2022nested} is that $A = [A_1, \cdots, A_n]$ has full row rank. Although this condition is much weaker compared to its predecessors, such as $A_i$ is the identity matrix \cite{zhang2020distributed} or $A_i$ has full row rank \cite{chang2014multi}, it is still too strong to be satisfied by some practical optimization problems. For instance, in the vertical federated learning setting for regression problems \cite{li2022nested}, where $A$ represents the feature matrix with each row corresponding to a sample and $A_i$ represents the local feature matrix of agent $i$, the number of samples is typically much larger than the number of features. As a result, $A$ fails to satisfy the full row rank condition. Hence, there is a need for algorithms that can achieve linear convergence under a weaker condition on $A_i$.

The major contributions of this work are summarized as follows.
\begin{enumerate}
  \item We propose iDDGT, a novel inexact decentralized dual gradient tracking method. Unlike existing algorithms that rely on either the exact dual gradient or an inexact one obtained through single-step gradient descent, iDDGT introduces a new approach: utilizing an inexact dual gradient
        with controllable levels of inexactness. Specifically, in iDDGT, the subproblem is approximately solved with a predefined accuracy during each iteration to regulate the level of inexactness in the dual gradient. It is proved that iDDGT can achieve linear convergence if the error in solving the subproblem decreases linearly.
  \item Thanks to the new approach for handling the dual gradient, iDDGT offers two significant advantages. Firstly, the inexactness of the dual gradient in each iteration can be controlled arbitrarily. This allows for adjusting the algorithm's overall performance by modifying the level of inexactness in different iterations. Secondly, the choice of the subproblem solver is flexible, enabling the utilization of accelerated or second-order methods to enhance the algorithm's overall performance. In numerical experiments, we compare the performances of iDDGT and NPGA, which is considered state-of-the-art \cite{li2022nested}. The results demonstrate that iDDGT achieves a significantly faster convergence speed in terms of the number of gradient steps compared to NPGA. Therefore, when computation is expensive but communication is cheap, iDDGT would be a preferable choice.
  \item A consequence of the above two advantages of iDDGT is that we can obtain multiple versions of iDDGT by choosing different subproblem solvers (such as single-step gradient descent, multi-step gradient descent, and AGD) and different strategies to control the level of inexactness in the dual gradient during different iterations. We compares the performances of different versions of iDDGT and observe some important facts. Firstly, it is an incredibly counterintuitive fact that using the exact gradient results in significantly lower computational and communication efficiencies, as measured by the number of gradient steps and communication rounds required to achieve a certain level of accuracy, compared to using an inexact dual gradient. Secondly, accelerating the reduction of subproblem solving errors within a certain range can enhance communication efficiency. However, it may also lead to a potential decrease in computational efficiency. Thirdly, employing single-step gradient descent as the subproblem solver can yield favorable communication efficiency. However, the computational efficiency is significantly lower compared to the strategy of linearly reducing the error in solving the subproblem.
  \item As mentioned earlier, some existing algorithms such as IDEA \cite{li2022nested}, DCPA \cite{alghunaim2021dual}, and NPGA can linearly solve \cref{original_pro} under the condition that $A$ has full row rank, which was the weakest condition prior to the introduction of iDDGT. However, iDDGT achieves linear convergence without imposing any conditions on $A_i$ or $A$, significantly expanding its applicability. Furthermore, the linear convergence of IDEA, DCPA, and NPGA (under the condition that $A$ has full row rank) is dependent on undirected graphs, whereas iDDGT can work for directed graphs.
\end{enumerate}

\section{Preliminaries} \label{preliminaries}
\textit{Notations:} $\1_n$ and $0_n$ represent the all-zero vector and the all-one vector, respectively, and $\mI_n$ denotes the $n\times n$ identity matrix. Notice that if the dimensions of the vectors consisting of ones and zeros, and the identity matrix can be inferred from the context, we will not explicitly indicate their dimensions. For $B \in \mR^{m \times n}$, $\lt[B\rt]_{ij}$ denotes the element of $B$ in the $i$-th row and the $j$-th column, $\us(B)$ and $\os(B)$ denote the smallest non-zero and largest singular values of $B$, respectively. $\Col(B)$ represents the column space of $B$.
$\norm{\cdot}$ denotes the Euclidean norm, and $\text{diag}(\cdot)$ denotes the (block) diagonal matrix.
For a vector $v \in \mR^n$, we define $\1v = \1 \otimes v$, where the dimension of $\1$ can be easily inferred from the context of $\1v$.

In the following, we provide several useful lemmas that will be utilized in the subsequent convergence analysis. \emph{In particular, if a lemma or theorem is not referenced and is not immediately followed by a proof, we assume that its proof is included in the appendix.}
\begin{lemma} \cite[Theorem 2.1.10]{nesterov2018lectures} \label{convex_property}
  Let $f:\mR^n \rightarrow \mR$ be continuously differentiable and $\mu$-strongly convex over $\mR^n$, then we have
  \eqe{
    \mu\norm{x-y} \leq \norm{\nabla f(x) - \nabla f(y)}, \ \forall x, y \in \mR^n.
    \nonumber
  }
\end{lemma}

\begin{lemma} \label{average}
  For any $B \in \mR^{nm \times q}$ ($B$ can be a vector with $q=1$), let it be partitioned as $B = [B_1\T, \cdots, B_n\T]\T$, we have
  $$\norm{B - \1\frac{1}{n}\sum_{i=1}^n B_i} \leq \norm{B}.$$
\end{lemma}

\section{Algorithm Design} \label{design}
\cref{original_pro} can be reformulated as
\begin{equation} \label{mP} \tag{P2}
  \begin{aligned}
    \min_{\mx \in \mR^d} \  & f(\mx)    \\
    \text{s.t.} \           & A\mx = b,
  \end{aligned}
\end{equation}
where $\mx = [x_1\T, \cdots, x_n\T]\T \in \mR^d$, $d=\sum_{i=1}^n d_i$, $f(\mx) = \sum_{i=1}^n f_i(x_i)$, and $A = [A_1, \cdots, A_n] \in \mR^{p \times d}$. In this work, we assume the following assumption holds.
\begin{assumption} \label{convex}
  $f_i$ is $\mu_i$-strongly convex and $l_i$-smooth over $\mR^{d_i}$, where $\mu_i$ and $l_i$ are both positive constants, $i = 1, \cdots, n$.
\end{assumption}

Since the strong duality holds for \cref{mP}, we can alternatively solve its dual. The dual function of \cref{mP} can be decomposed as
\eqe{
  \phi(\lambda) &= \inf_{\mx \in \mR^d} f(\mx) + \lambda\T(A\mx - b) \\
  &= \sum_{i=1}^n\inf_{x_i \in \mR^{d_i}} f_i(x_i) + \lambda\T\lt(A_ix_i - \frac{1}{n}b\rt) \\
  &= \sum_{i=1}^n \phi_i(\lambda),
  \nonumber
}
then we can reformulate the dual of \cref{mP} as
\begin{equation} \label{dual_pro} \tag{P3}
  \max_{\lambda \in \mR^{p}} \ \phi(\lambda).
\end{equation}
An important fact is that if $f$ is strictly convex, then $\phi$ is differentiable and \cite{bertsekas1997nonlinear}
\eqe{
  \nabla \phi(\lambda) = A\mx^*(\lambda) - b,
}
where $\mx^*(\lambda) = \arg\min_{\mx \in \mR^{d}}\lt\{f(\mx) + \lambda\T(A\mx - b)\rt\}$. Given \cref{convex}, obviously $f$ is $\mu$-strongly convex and $l$-smooth, where $\mu = \min_{i = 1, \cdots, n}{\mu_i}$ and $l = \max_{i = 1, \cdots, n}{l_i}$. Then we have the following lemma.
\begin{lemma} \label{dual_property}
  Suppose \cref{convex} holds, then $\phi$ is $\frac{\os^2(A)}{\mu}$-smooth over $\mR^p$ and $\frac{\us^2(A)}{l}$-strongly concave over $\Col(A)$.
\end{lemma}

\begin{remark}
  An evident fact about \cref{dual_property} is that the solution of \cref{dual_pro}, denoted as $\lambda^*$, is not unique unless $A$ has full row rank, which ensures that $\phi$ is strongly concave over $\mR^p$. However, \cref{dual_property} also indicates that $\phi$ is strongly concave over $\Col(A)$, implying that the projection of $\lambda^*$ onto $\Col(A)$, denoted as $\lambda^*_c$, is unique.
\end{remark}

Applying the classical gradient method to \cref{dual_pro} gives the dual ascent method (DA)
\eqe{
  \lambda\+ = \lambda^k + \alpha\nabla\phi(\lambda^k),
  \nonumber
}
which can be unfolded as
\eqe{
  \mx\+ &= \arg\min_{\mx \in \mR^{d}}\lt\{f(\mx) + \lambda^{k\T}(A\mx - b)\rt\}, \\
  \lambda\+ &= \lambda^k + \alpha\lt(A\mx\+ - b\rt),
}

The following lemma illustrates the contraction property of DA, which also implies its linear convergence.
\begin{lemma} \label{contraction}
  Suppose \cref{convex} holds, $\lambda^0 = 0$, and the step-size satisfies $0 < \alpha < \frac{2\mu}{\os^2(A)}$, then we have
  \eqe{
    \norm{\lambda\+ - \lambda^*_c} \leq \eta\norm{\lambda^k - \lambda^*_c}, \ \forall k \geq 0,
  }
  where $\eta = \max\lt\{\lt|1-\frac{\alpha\os^2(A)}{\mu}\rt|, \lt|1-\frac{\alpha\us^2(A)}{l}\rt|\rt\} \in (0, 1)$.
\end{lemma}

\begin{algorithm}[tb]
  \caption{{\small Inexect Decentralized Dual Gradient Tracking}}
  \label{alg:iD2GT}
  \begin{algorithmic}[1]
    \Require
    $W$, $K$, $\beta$, the subproblem solver
    \Ensure
    $x_i^K$, for $i= 1, \cdots, n$
    \State Agent $i$ implements
    \State $x_i^0 = 0$, $z_i^0 = -\frac{1}{n}b$
    \For {$k=0,\dots, K-1$ }
    \State Solve
    \eqe{ \label{inner_problem}
      \min_{x_i \in \mR^{d_i}}\lt\{f_i(x_i) + \lambda_i^{k\T}\lt(A_ix_i - \frac{1}{n}b\rt)\rt\}
    }
    to obtain an inexact solution $x_i\+$.
    \State $z_i\+ = \sum_{j=1}^n[W]_{ij}z_j^k + A_i(x_i\+-x_i^k)$
    \State $\lambda_i\+ = \sum_{j=1}^n[W]_{ij}\lt(\lambda_j^k + \beta z_j\+\rt)$
    \EndFor
  \end{algorithmic}
\end{algorithm}

While DA can achieve linear convergence in solving \cref{dual_pro}, it cannot be implemented in a decentralized manner due to the requirement of global information for the dual gradient $\sum_{i=1}^n A_i x_i - b$. Additionally, solving a subproblem exactly to obtain the dual gradient $\nabla \phi(\lambda^k)$ at each iteration of DA is computationally expensive and often impractical.
To address these limitations, we propose iDDGT, which is a decentralized version of DA that eliminates the need for solving the subproblem exactly at each iteration.

Let $\ml = [\lambda_0\T, \cdots, \lambda_k\T]\T$, $\mz = [z_0\T, \cdots, z_k\T]\T$, $\mA = \text{diag}(A_1, \cdots, A_n)$, $\mb = \1_n \otimes \frac{1}{n}b$, and $\mW = W \otimes \mI_p$, we can rewrite iDDGT in a compact form as follows.
\begin{subequations} \label{id2gt}
  \begin{align}
    \mx\+ & \approx \arg\min_{\mx \in \mR^{d}}\lt\{f(\mx) + \ml^{k\T}\lt(\mA\mx-\mb\rt)\rt\}, \label{id2gt_x} \\
    \mz\+ & = \mW\mz^k + \mA(\mx\+-\mx^k), \label{id2gt_z}                                                    \\
    \ml\+ & = \mW\lt(\ml^k + \beta\mz\+\rt). \label{id2gt_l}
  \end{align}
\end{subequations}

\begin{remark}
  The decentralized nature of iDDGT originates from the classical gradient tracking technique \cite{nedic2017achieving,qu2017harnessing,scutari2019distributed}, which is utilized to track the global dual gradient in a decentralized manner. The key features of iDDGT lie in the flexibility to control the level of inexactness in the dual gradient and the freedom to choose the subproblem solver. Thanks to these features, iDDGT demonstrates significantly higher computational efficiency compared to state-of-the-art methods in numerical experiments. Moreover, it has been proven that iDDGT can achieve linear convergence over directed graphs without imposing any conditions on $A_i$ or $A$. In contrast, existing algorithms require $A$ to have full row rank and the graphs to be undirected in order to achieve similar convergence guarantees.
\end{remark}

\begin{remark}
  A related work is \cite{zhang2020distributed}, which considers a special case of \cref{original_pro} where $A_i = \mathbf{I}$, and proposes a similar algorithm called distributed dual gradient tracking (DDGT). The main distinctions between iDDGT and DDGT lie in the range of problem settings they can handle and the approach they employ for utilizing the dual gradient. DDGT is only capable of solving \cref{original_pro} when $A_i = \mI$, while iDDGT can handle general $A_i$, giving it much a broader range of applications. It is worth noting that this generalization is non-trivial since the convergence analysis of DDGT heavily relies on the property $A_i = \mI$, making it difficult to extend to general $A_i$. Furthermore, DDGT utilizes the exact dual gradient and solves the subproblem exactly at each iteration, which could be computationally expensive and impractical. Moreover, as mentioned earlier, we have observed that using the exact gradient leads to significantly lower computational and communication efficiencies compared to using an inexact dual gradient. Consequently, iDDGT is much more efficient than DDGT. The advantage of DDGT lies in its ability to work over more general directed graphs compared to iDDGT.
\end{remark}

\section{Convergence Analysis} \label{convergence}
In this section, we analyze the linear convergence of iDDGT.
\begin{assumption} \label{W}
  The mixing matrix $W$ associated with the network graph is assumed to be primitive, doubly stochastic, and with positive diagonal entries.
\end{assumption}
\begin{remark}
  \cref{W} can be satisfied by strongly-connected directed graphs that admit doubly-stochastic weights (see \cite{xin2022fast} for a more detailed discussion), which can cover connected undirected graphs as special cases. Consequently, our network condition is more inclusive compared to \cite{li2022implicit,li2022nested,alghunaim2021dual}, where only undirected graphs are considered.
\end{remark}
Given \cref{W}, $W$ possesses the following important property \cite{qu2017harnessing,xin2022fast}\footnote{Though the above property is derived under the assumption that $\mathcal{G}$ is undirected and connected in \cite{qu2017harnessing}, it can be trivially proven for our case using the Perron-Frobenius theory.}
\eqe{ \label{sigma}
  \sigma =& \norm{W-\frac{1}{n}\1\1\T} \in (0,1),
}
then the following lemma immediately holds.
\begin{lemma}[\cite{qu2017harnessing}] \label{consensus}
  Suppose \cref{W} holds, then we have
  $$\norm{Wx - \frac{1}{n}\1\1\T x} \leq \sigma\norm{x-\frac{1}{n}\1\1\T x}, \ \forall x \in \mR^n.$$
\end{lemma}

Let
\eqe{
  \mx^*(\ml) &= \arg\min_{\mx \in \mR^{d}}\lt\{f(\mx) + \ml\T\lt(\mA\mx-\mb\rt)\rt\},
}
and define
\eqe{
  \az^k = \frac{1}{n}\sum_{i=1}^n z_i^k, \ \al^k = \frac{1}{n}\sum_{i=1}^n \lambda_i^k, \ax^k = \mx^*\lt(\1\al^{k-1}\rt),
}
then we have the following lemma.
\begin{lemma} \label{average_z}
  Given $z_i^0 = A_ix^0_i - \frac{1}{n}b$, then we have
  \eqe{
    \az\+ &= \az^k + \frac{1}{n}A\lt(\mx\+-\mx^k\rt) = \frac{1}{n}(A\mx\+-b), \\
    \al\+ &= \al^k + \beta\az\+ = \al^k + \frac{\beta}{n}(A\mx\+-b).
    \nonumber
  }
\end{lemma}
The following lemma establishes a linear matrix inequality regarding the iterations of iDDGT, which is crucial for proving its linear convergence.
\begin{lemma} \label{key_lemma}
  Suppose \cref{convex,W} holds, $\ml^0 = 0$, the step-size satisfies $0 < \beta < \frac{2n\mu}{\os^2(A)}$, and
  \eqe{ \label{inner_error}
    \norm{\mx\+-\mx^*(\ml^k)} \leq \delta\+, \ \forall k \geq 0,
  }
  then we have
  \eqe{
  \zeta^k \leq M^k\zeta^0 + \sum_{i=0}^{k-1}M^{k-1-i}H\xi^i, \ \forall k \geq 1,
  \nonumber
  }
  where $\zeta^k = \lt[\begin{array}{c}
        \norm{\mz^k-\1\az^k} \\
        \norm{\ml^k-\1\al^k} \\
        \norm{\ml^k-\ml\p}   \\
        \sqrt{n}\norm{\al^k-\lambda^*}
      \end{array}\rt]$, $\xi^k = \lt[\delta\+, \delta^k\rt]\T$,
  \eqe{ \label{2315}
    M = \left[\begin{array}{ccccc}
        \sigma        & 0                                                & \frac{\os^2(\mA)}{\mu}            & 0                          \\
        \beta\sigma^2 & \sigma                                           & \frac{\beta\sigma\os^2(\mA)}{\mu} & 0                          \\
        \beta\sigma   & 1+\sigma+\frac{\beta\os(A)\os(\mA)}{\sqrt{n}\mu} & \frac{\beta\os^2(\mA)}{\mu}       & \frac{\beta\os^2(A)}{n\mu} \\
        0             & \frac{\beta\os(A)\os(\mA)}{\sqrt{n}\mu}          & 0                                 & \nu
      \end{array}\right]
    \nonumber
  }
  and
  \eqe{ \label{2316}
    H = \left[\begin{array}{ccc}
        \os(\mA)                                      & \os(\mA)            \\
        \beta\sigma\os(\mA)                           & \beta\sigma\os(\mA) \\
        \beta\lt(\os(\mA)+\frac{\os(A)}{\sqrt{n}}\rt) & \beta\os(\mA)       \\
        \frac{\beta\os(A)}{\sqrt{n}}                  & 0
      \end{array}\right].
    \nonumber
  }
\end{lemma}

\begin{algorithm}[tb]
  \caption{{\small Nesterov's Accelerated Gradient Descent}}
  \label{alg:AGD}
  \begin{algorithmic}[1]
    \Require
    $g$ ($\mu$-strongly convex and $l$-smooth), $\kappa = \frac{l}{\mu}$, $x^0$
    \Ensure
    $x^T$
    \State $y^0 = x^0$
    \For {$k=0,\dots, T-1$ }
    \State $x\+ = y^k - \frac{1}{l}\nabla g(y^k)$
    \State $y\+ = x\+ + \frac{\sqrt{\kappa}-1}{\sqrt{\kappa}+1}\lt(x\+-x^k\rt)$
    \EndFor
  \end{algorithmic}
\end{algorithm}

\begin{figure*}[tb]
  \begin{center}
    \includegraphics[scale=0.45]{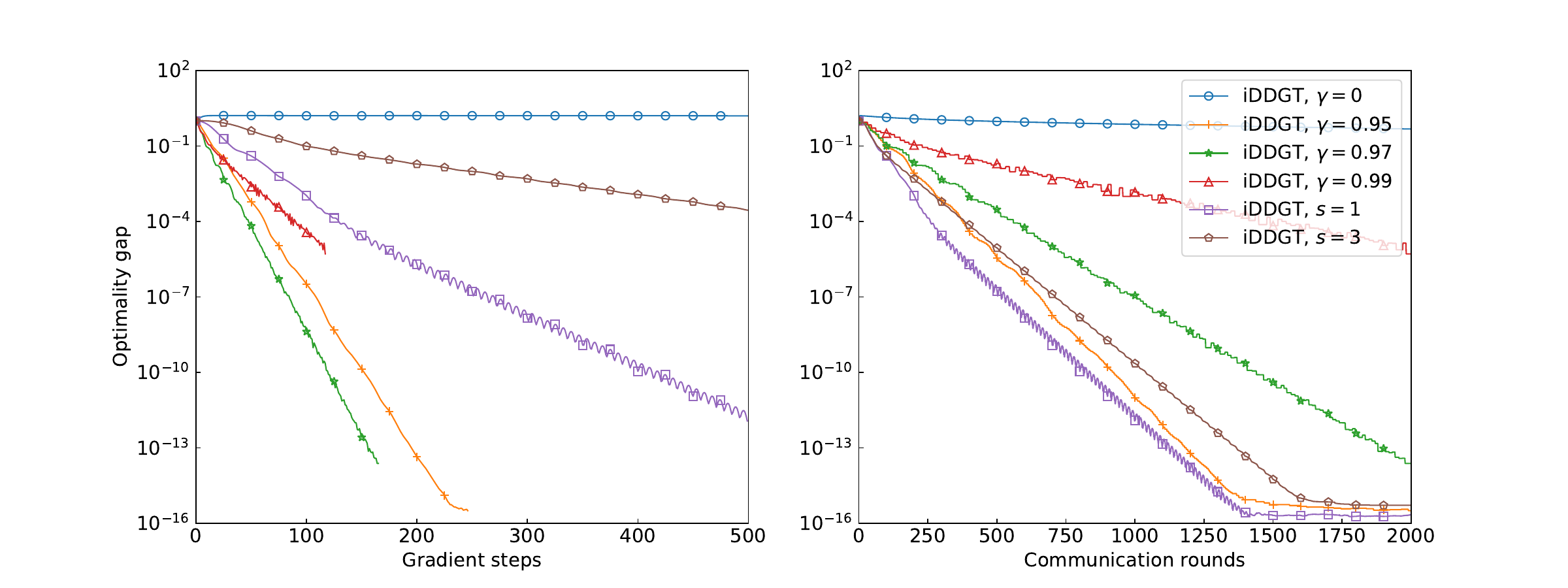}
    \caption {The result of Experiment I.}
    \label{fig1}
  \end{center}
\end{figure*}

\begin{remark}
  There are various strategies to solve the subproblem \cref{id2gt_x} and obtain an inexact solution that satisfies the suboptimality condition \cref{inner_error}. To ensure \cref{inner_error}, we can simply require agent $i$ to satisfy
  \eqe{ \label{inner_error_i}
    \norm{x_i\+-x_i^*(\lambda_i^k)} \leq \frac{\delta\+}{\sqrt{n}}.
  }
  Typically, we employ unconstrained optimization algorithms such as gradient descent, AGD, or second-order methods to iteratively solve \cref{inner_error_i}. Hence, a straightforward approach is to set a stopping condition that is sufficient for \cref{inner_error_i} and terminate the iteration once the condition is met.
  Let $F_i^k(x_i) = f_i(x_i) + \lambda_i^{k\T}\lt(A_ix_i-\frac{1}{n}b\rt)$, obviously $F_i^k(x_i)$ is $\mu_i$-strongly convex and $l_i$-smooth for $k \geq 0$. Notice that $x_i^*(\lambda_i^k)$ represents the solution of $\min_{x_i \in \mR^{d_i}}F_i^k(x_i)$, implying that $\nabla F_i^k(x_i^*(\lambda_i^k)) = 0$. Consequently, we have
  \eqe{
    \norm{x_i\+-x_i^*(\lambda_i^k)} &\leq \frac{1}{\mu}\norm{\nabla F_i^k(x_i\+)} \\
    &= \frac{1}{\mu}\norm{\nabla f(x_i\+) + A_i\T\lambda_i^k}.
  }
  Thus, the stopping condition for agent $i$ to ensure \cref{inner_error_i} can be expressed as
  \eqe{
    \norm{\nabla f(x_i\+) + A_i\T\lambda_i^k} \leq \frac{\mu\delta\+}{\sqrt{n}}.
  }
  Another approach is to predefine the number of inner iterations, which can be estimated based on the theoretical convergence rate of the selected algorithm. \cref{inner_iteration} provides a lower bound on the number of inner iterations for AGD.
\end{remark}

\begin{lemma} \label{inner_iteration}
  Suppose \cref{convex} holds, and AGD (i.e., \cref{alg:AGD}) is chosen as the solver for the inner problem \cref{inner_problem}, then agent $i$ requires at least $\sqrt{\frac{l_i}{\mu_i}}\ln\lt(\frac{n(l_i+\mu_i)\norm{\nabla f_i\lt(x_i^{k+1, 0}\rt)}^2}{(\delta\+)^2\mu^3}\rt)$ inner iterations to ensure \cref{inner_error} at the $k$-th outer iteration, where $x_i^{k+1, 0}$ is the chosen initial value of AGD, $i = 1, \cdots, n$.
\end{lemma}

In the following theorem, we show that iDDGT can achieve linear convergence if $\delta^k$ decreases linearly.
\begin{theorem} \label{main_theorem}
  Suppose \cref{convex,W} holds, $\ml^0 = 0$, the step-size satisfies
  \eqe{ \label{step_condition}
    0 < \beta < \max\lt\{\frac{\mu}{\os^2(\mA)}, \frac{(1-\sigma)^2\us(A)^2\mu^2}{16\os^4(\mA)nl}\rt\},
  }
  and $\delta^k$ defined in \cref{inner_error} satisfies
  \eqe{ \label{error_con}
    \delta\+ = \gamma\delta_k, \ \forall k \geq 0,
  }
  with $\gamma \in (0, 1)$, then we have
  $$\norm{\mx^k-\mx^*} = \mathcal{O}\lt(\theta^k\rt),$$
  where
  $$\theta = \max\lt(1-\frac{\beta\us^2(A)}{2nl}, \sigma+4\sqrt{\beta\frac{\os^4(\mA)nl}{\us^2(A)\mu^2}}, \gamma\rt) \in (0, 1).$$
\end{theorem}

\begin{proof}
  Notice that $M$ is a nonnegative and irreducible matrix, then we have
  $$\lt[M^k\rt]_{ij} = \mathcal{O}\lt(\rho(M)^k\rt), \ i, j = 1, \cdots, 4.$$
  \cref{key_lemma} states that
  \eqe{
  \zeta^k \leq M^k\zeta^0 + \sum_{i=0}^{k-1}M^{k-1-i}H\xi^i,
  \nonumber
  }
  combining it with \cref{error_con} gives that
  \eqe{
  \zeta^k \leq M^k\zeta^0 + \frac{1}{\gamma}\sum_{i=0}^{k-1}\gamma^{i+1}M^{k-1-i}H\xi^0.
  \nonumber
  }
  It follows that
  \eqe{
    \norm{\ml^k-\1\al^k} &= \mathcal{O}\lt(\max\lt(\rho(M), \beta\rt)^k\rt), \\
    \norm{\al^k-\lambda^*} &= \mathcal{O}\lt(\max\lt(\rho(M), \beta\rt)^k\rt),
    \nonumber
  }
  then we have
  \eqe{
    &\norm{\mx\+-\mx^*} \\
    \leq& \norm{\mx\+-\mx^*(\ml^k)} + \norm{\mx^*(\ml^k)-\mx^*} \\
    \leq& \gamma\+\delta^0 + \frac{\os(\mA)}{\mu}\lt(\norm{\ml^k-\1\al^k} + \sqrt{n}\norm{\al^k-\lambda^*}\rt) \\
    =& \mathcal{O}\lt(\max\lt(\rho(M), \gamma\rt)\+\rt).
  }

  Now we need to find an upper bound for $\rho(M)$. Let $a_1=\frac{\os(\mA)}{\sqrt{\mu}}$ and $a_2=\frac{\os(A)}{\sqrt{n\mu}}$, then the characteristic polynomial of $M$ is given as
  \eqe{
    p(x) = (x-\eta)\lt[xp_0(x)-(a_1a_2\beta)^3\sigma\rt]-(a_1a_2\beta)^3\sigma\eta,
    \nonumber
  }
  where
  $$p_0(x)=x^2-\lt(a_1^2\beta+2\sigma\rt)x-\lt(a_1^3a_2\beta^2\sigma+a_1^2\beta\sigma^2-\sigma^2\rt).$$

  Note that \cref{step_condition} guarantees that $\beta<\frac{\mu}{\os^2(\mA)}$, and it holds that
  $$\os(A) = \norm{A} = \norm{\lt(\1_n \otimes \mI_p\rt)\mA} \leq \sqrt{n}\os(\mA).$$
  Therefore, we have
  \eqe{
    a_2^2\beta \leq a_1a_2\beta \leq a_1^2\beta \leq 1,
  }
  which implies that the two roots of $p_0$ satisfy
  \eqe{
    \frac{1}{2}\lt(a_1^2\beta+2\sigma+\sqrt{(a_1^2\beta)^2+4(a_1^2\beta\sigma+a_1^3a_2\beta^2\sigma+a_1^2\beta\sigma^2)}\rt) \\
    < \sigma+3\sqrt{a_1^2\beta}.
    \nonumber
  }
  Consequently, it follows that
  \eqe{
    p_0(x) \geq \lt(x-\sigma-3\sqrt{a_1^2\beta}\rt)^2, \ \forall x \geq \sigma+3\sqrt{a_1^2\beta}.
  }

  Let
  \eqe{
    \hat{x} = \max\lt\{1-\frac{\beta\us^2(A)}{2nl}, \sigma+4\sqrt{a_1^2\beta}\sqrt{\frac{\os^2(\mA)nl}{\us^2(A)\mu}}\rt\},
  }
  note that $\frac{\os^2(\mA)nl}{\us^2(A)\mu} \geq 1$, then we have
  \eqe{
    \hat{x} \geq \sigma+4\sqrt{a_1^2\beta}\sqrt{\frac{\os^2(\mA)nl}{\us^2(A)\mu}} \geq \sigma+3\sqrt{a_1^2\beta} \geq 3a_1^2\beta,
  }
  and
  \eqe{
    p_0(\hat{x}) \geq a_1^2\beta\frac{\os^2(\mA)nl}{\us^2(A)\mu}.
  }
  Also note that $\beta<\frac{n\mu}{\os^2(A)}$, then we have $\eta = 1-\frac{\beta\us^2(A)}{nl} < 1$.
  It follows that
  \eqe{
    p(\hat{x}) \geq& \frac{\beta\us^2(A)}{2nl}\lt[\hat{x}p_0(\hat{x})-(a_1^2\beta)^3\rt]-(a_1^2\beta)^3 \\
    \geq& \frac{\beta\us^2(A)}{2nl}\lt[\frac{3(a_1^2\beta)^2\os^2(\mA)nl}{\us^2(A)\mu}-(a_1^2\beta)^3\rt]-(a_1^2\beta)^3 \\
    \geq& (a_1^2\beta)^3-(a_1^2\beta)^3 \\
    \geq& 0,
  }
  which implies that $p(x)$ is monotone increasing on $[\hat{x}, +\infty)$, hence all real roots of $p(x)$ lie in $(-\infty, \hat{x}]$. According to the Perron-Frobenius theorem, we know that $\rho(M)$ is an eigenvalue of $M$, hence $\rho(M) \leq \hat{x} = \max\lt\{1-\frac{\beta\us^2(A)}{2nl}, \sigma+4\sqrt{a_1^2\beta}\sqrt{\frac{\os^2(\mA)nl}{\us^2(A)\mu}}\rt\}$. Also note that $\rho(M) < 1$ if \cref{step_condition} holds, then the proof is completed.
\end{proof}

\begin{remark}
  As demonstrated in \cref{main_theorem}, iDDGT can achieve linear convergence over directed graphs without imposing any conditions on $A_i$ or $A$. In contrast, existing algorithms such as IDEA, DCPA, and NPGA require $A$ to have full row rank and are limited to undirected graphs. Thus, iDDGT has a much broader scope of application than these algorithms. Furthermore, in numerical experiments where $A$ has full row rank and the graph is undirected, iDDGT exhibits significantly faster convergence in terms of the number of gradient steps compared to NPGA, which is considered state-of-the-art.
\end{remark}

\begin{figure*}[tb]
  \begin{center}
    \includegraphics[scale=0.45]{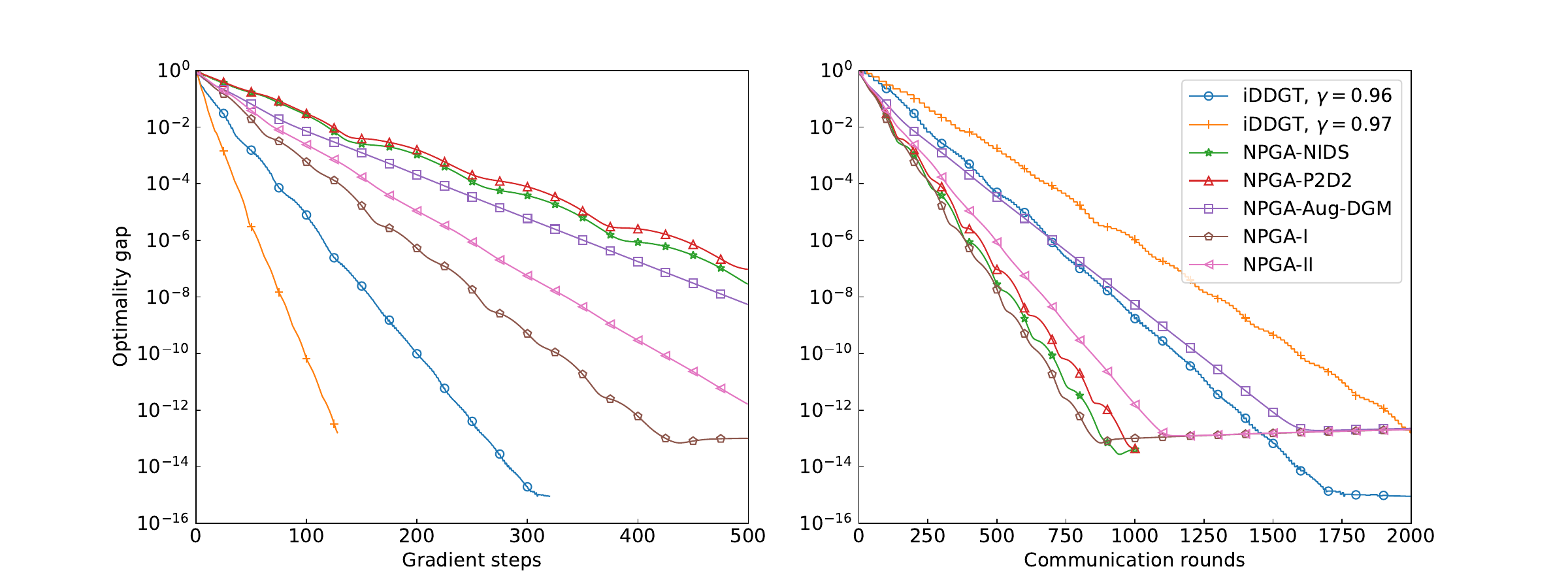}
    \caption {The result of Experiment II.}
    \label{fig2}
  \end{center}
\end{figure*}
\section{Numerical Experiments} \label{experiments}
In this section, we take two numerical experiments to validate the theoretical results and compare the performance of iDDGT with existing algorithms.

\subsection{Experiment I}
This experiment aims to validate \cref{main_theorem}, which states that iDDGT can achieve linear convergence for solving \cref{original_pro} over directed graphs, even if the matrix $A$ does not have full row rank. We consider the following instance of \cref{original_pro}:
\eqe{ \label{expt_pro}
  \min_{x_i \in \mR^{d_i}} \  & \sum_{i=1}^n \frac{1}{2}x_i\T P_ix_i + q_i\T x_i    \\
  \text{s.t.} \               & \sum_{i=1}^{n}A_ix_i = b,
}
where $P_i \in \mR^{d_i \times d_i}$ is a positive definite matrix. In this experiment, we choose $n=20$ and generate a directed exponential graph with $20$ nodes using a parameter $e = 4$. For a directed exponential graph, node $i$ can send messages to nodes $((i + 2^j) \mod n)$ for $j = 0, 1, \cdots, e$.
The matrix $P_i \in \mR^{2 \times 2}$ is randomly generated using the reverse process of diagonal decomposition, ensuring that their eigenvalues belong to the interval $[1, 10]$. Each element of fhe first $20$ rows of $A_i \in \mR^{100 \times 2}$ is independently sampled from a normal distribution with mean $0$ and variance $10$, while the remaining $80$ rows are generated by linearly combining the first $20$ rows. As a result, the row rank of the final matrix $A \in \mR^{100 \times 2}$ is $20$, indicating that it does not have full row rank. Each element of $q_i \in \mR^2$ and $b \in \mR^{100}$ is independently sampled from a standard normal distribution.

We compare the performances of different versions of iDDGT, with AGD chosen as the subproblem solver. The variations in iDDGT versions lie in the solving strategy of subproblems. One strategy involves controlling the solving error and decreasing the error linearly with respect to outer iterations, as described in \cref{main_theorem}. The other strategy involves using a fixed number of inner iterations. For example, "iDDGT, $\gamma=0.95$" denotes a version of iDDGT that uses the first strategy with an error decreasing rate of $0.95$, and "iDDGT, $s=1$" denotes a version that uses the second strategy with a fixed number of inner iterations set to $1$.

The experiment result is shown in \cref{fig1}, where the optimality gap is defined as $\frac{\|\mx^k-\mx^*\|}{\|\mx^0-\mx^*\|}$. Several observations can be made:
\begin{enumerate}
  \item The versions of iDDGT that adopt the first strategy demonstrate linear convergence, which confirms the validity of \cref{main_theorem}.
  \item Within a certain range, accelerating the reduction of subproblem solving errors in the first strategy can enhance communication efficiency. However, it may also lead to a potential decrease in computational efficiency. An important and counterintuitive finding is that solving the subproblem exactly (i.e., $\gamma = 0$) can result in both low computational efficiency and communication efficiency.
  \item Using single-step gradient descent as the subproblem solver can yield favorable communication efficiency. However, the computational efficiency is significantly inferior compared to the first strategy.
\end{enumerate}

\begin{remark}
  Though the fact that solving the subproblem exactly could result in both low computational efficiency and communication efficiency is counterintuitive, it can still be understood. The direct reason is that when the subproblem is solved inexactly, the value of $\beta$ could be much larger compared to when solving the subproblem exactly in experiments. An intuitive explanation for this is that when solving the subproblem inexactly, even if a larger value of $\beta$ is used (which could potentially lead to divergence if solving the subproblem exactly), $\mx\+$ would not be pulled very far from the convergent sequence. This is because only a few iterations are taken to solve the subproblem, which preserves the possibility of convergence.
\end{remark}

\subsection{Experiment II}
In this experiment, we continue to consider the optimization problem \cref{expt_pro} but with slightly different settings. As mentioned earlier, IDEA, DCPA, and NPGA can achieve linear convergence over undirected graphs under the condition that matrix $A$ has full row rank, with NPGA showing the best performance in numerical experiments \cite{li2022nested}. Therefore, in this comparison, we focus on evaluating the performances of iDDGT and NPGA. Since NPGA is an algorithmic framework with various variants, we have selected some of its best-performing variants, namely NPGA-NIDS, NPGA-P2D2, NPGA-Aug-DGM, NPGA-I, and NPGA-II.

To ensure a fair comparison between iDDGT and NPGA, we need to use the setting where the graph is undirected and matrix $A$ has full row rank. Specifically, we adopt the same settings and data as in Experiment I, with the exception of the graph, $A_i$, and $b$. The undirected graph with 20 nodes is generated using the Erdos-Renyi model \cite{erdos1960evolution} with a connectivity probability of 0.3. The elements of $A_i \in \mathbb{R}^{20 \times 2}$ and $b \in \mathbb{R}^{20}$ are randomly and independently sampled from normal distributions with mean $0$ and variance $10$ for $A_i$, and from the standard normal distribution for $b$. The resulting matrix $A$ is guaranteed to have full row rank.

The experiment result is shown in \cref{fig2}. We can observe that the convergence speed of iDDGT in terms of the number of gradient steps is much faster than that of NPGA, whereas its convergence speed in terms of the number of communication rounds is slower compared to NPGA. Therefore, iDDGT woule be a better choice when computation is expensive but communication is cheap.

\section{Conclusion} \label{conclusion}
In this work, we have presented iDDGT, an inexact decentralized dual gradient tracking method for distributed optimization problems with a globally coupled equality constraint. By utilizing an inexact dual gradient with controllable inexactness, iDDGT offers significant computational efficiency advantages over existing algorithms.
Another key contribution of iDDGT is its ability to achieve linear convergence over directed graphs without imposing any conditions on the constraint matrix. This significantly broadens the scope of its applicability compared to existing algorithms that require the constraint matrix to have full row rank and undirected graphs for linear convergence.

Overall, iDDGT offers a promising approach for solving constraint-coupled optimization problems. Its ability to achieve linear convergence, computational efficiency, and flexibility make it a valuable tool for a wide range of applications. Future research can focus on extending iDDGT to handle more complex constraints and exploring adaptive inexactness control.

\bibliographystyle{ieeetr}
\bibliography{../public/bib}

\begin{appendix} \label{appendix}
  \begin{appendix_proof}[\cref{average}] Note that
    \eqe{
      &\norm{B - \1\frac{1}{n}\sum_{i=1}^n B_i}^2 \\
      =& \norm{B}^2 + \norm{\1\frac{1}{n}\sum_{i=1}^n B_i}^2 - 2\sum_{i=1}^nB_i\T\lt(\frac{1}{n}\sum_{i=1}^n B_i\rt) \\
      =& \norm{B}^2 + n\norm{\frac{1}{n}\sum_{i=1}^n B_i}^2 - 2n\norm{\frac{1}{n}\sum_{i=1}^n B_i}^2, \\
      \leq& \norm{B}^2,
    }
    where completes the proof.
  \end{appendix_proof}

  \begin{appendix_proof}[\cref{dual_property}]
    Note that
    \eqe{
      \phi(\lambda) =& \inf_{\mx \in \mR^d} f(\mx) + \lambda\T(A\mx - b) \\
      =& -f^*(-A\T\lambda) - \lambda\T b,
    }
    then we have
    \eqe{ \label{1726}
      &\dotprod{\nabla \phi(\lambda_1)-\nabla \phi(\lambda_2), \lambda_1-\lambda_2} \\
      =& \dotprod{\nabla f^*(-A\T\lambda_1)-\nabla f^*(-A\T\lambda_2), A\T\lt(\lambda_1-\lambda_2\rt)}. \\
    }
    As mentioned before, \cref{convex} implies that $f$ is $\mu$-strongly convex and $l$-smooth, hence $f^*$ is $\frac{1}{l}$-strongly convex and $\frac{1}{\mu}$-smooth. Applying the smoothness of $f^*$ to \cref{1726} gives that
    \eqe{
      \dotprod{\nabla \phi(\lambda_1)-\nabla \phi(\lambda_2), \lambda_1-\lambda_2} \geq& -\frac{1}{\mu}\norm{A\T\lt(\lambda_1-\lambda_2\rt)}^2 \\
      \geq& -\frac{\os^2(A)}{\mu}\norm{\lambda_1-\lambda_2}^2,
    }
    hence $\phi(\lambda)$ is $\frac{\os^2(A)}{\mu}$-smooth. Applying the strong convexity of $f^*$ to \cref{1726} gives that
    \eqe{
      \dotprod{\nabla \phi(\lambda_1)-\nabla \phi(\lambda_2), \lambda_1-\lambda_2}  \leq& -\frac{1}{l}\norm{A\T\lt(\lambda_1-\lambda_2\rt)}^2 \\
      \leq& -\frac{\us^2(A)}{l}\norm{\lambda_1-\lambda_2}^2,
    }
    which the second inequality holds if $\lambda_1, \lambda_2 \in \Col(A)$. Therefore, $\phi(\lambda)$ is $\frac{\us^2(A)}{l}$-strongly concave over $\Col(A)$, which completes the proof.
  \end{appendix_proof}

  \begin{appendix_proof}[\cref{contraction}]
    Let $\mx^*$ be the solution of \cref{mP}, it holds that $A\mx^* = b$, then we have
    \eqe{
      \lambda\+ &= \lambda^k + \alpha A\lt(\mx\+ - \mx^*\rt).
      \nonumber
    }
    Recall that $\lambda^0 = 0$, hence $\lambda^k \in \Col(A), \ \forall k \geq 0$, which implies that we can use the strong concaveness of $\phi$ for them.

    The following proof is borrowed from \cite[Lemma 10]{qu2017harnessing}, where the contraction property of the gradient method is studied.
    We first consider the case $0 < \alpha \leq \frac{2}{\frac{\os^2(A)}{\mu} + \frac{\us^2(A)}{l}}$. Let $\mu' = \frac{\us^2(A)}{l}$ and $l' = \frac{2}{\alpha} - \frac{\us^2(A)}{l} \geq \frac{\os^2(A)}{\mu}$, then $\phi$ is also $\mu'$-strongly concave over $\Col(A)$ and $l'$-smooth.
    It follows that
    \eqe{
      &\norm{\lambda\+ - \lambda^*_c}^2 \\
      =& \norm{\lambda^k + \alpha\nabla\phi(\lambda^k) - \lt(\lambda^*_c + \alpha\nabla\phi(\lambda^*_c)\rt)}^2 \\
      =& \norm{\lambda^k-\lambda^*_c}^2 + 2\alpha\dotprod{\nabla\phi(\lambda^k)-\nabla\phi(\lambda^*_c), \lambda^k-\lambda^*_c} \\
      &+ \alpha^2\norm{\nabla\phi(\lambda^k)-\nabla\phi(\lambda^*_c)} \\
      \leq& \lt(1-2\alpha\frac{\mu'l'}{\mu'+l'}\rt)\norm{\lambda^k-\lambda^*_c}^2 \\
      &+ \alpha\lt(\alpha-\frac{2}{\mu'+l'}\rt)\norm{\nabla\phi(\lambda^k)-\nabla\phi(\lambda^*_c)} \\
      =& \lt(1-\frac{\alpha\us^2(A)}{l}\rt)^2\norm{\lambda^k-\lambda^*_c}^2 \\
      =& \eta^2\norm{\lambda^k-\lambda^*_c}^2,
    }
    where the inequality holds because of \cite[Lemma 3.11]{bubeck2015convex} and the last equality holds due to $\frac{\us^2(A)}{l} \leq \frac{\os^2(A)}{\mu}$ and $0 < \alpha \leq \frac{2}{\frac{\os^2(A)}{\mu} + \frac{\us^2(A)}{l}}$.
    The case $\frac{2}{\frac{\os^2(A)}{\mu} + \frac{\us^2(A)}{l}} < \alpha < \frac{2\mu}{\os^2(A)}$ can proved in a similar way with $\mu' = \frac{2}{\alpha} - \frac{\os^2(A)}{\mu}$ and $l' = \frac{\os^2(A)}{\mu}$.
  \end{appendix_proof}

  \begin{appendix_proof}[\cref{average_z}]
    Note that $z_i^0 = A_i(\mx^0_i) - \frac{1}{n}b$ implies that $\az^0 = \frac{1}{n}(A\mx^0-b)$, then applying the double stochasticity of $\mW$ to $\cref{id2gt_z}$ gives that
    \eqe{
      \az\+ = \az^k + \frac{1}{n}A\lt(\mx\+-\mx^k\rt) = \frac{1}{n}(A\mx\+-b),
      \nonumber
    }
    it follows that
    \eqe{
      \al\+ &= \al^k + \beta\az\+ = \al^k + \frac{\beta}{n}(A\mx\+-b),
      \nonumber
    }
    which completes the proof.
  \end{appendix_proof}

  \begin{appendix_proof}[\cref{key_lemma}]
    Recall that $f$ is $\mu$-strongly convex, according to \cref{convex_property}, we have
    \eqe{ \label{1734}
      &\norm{\mx^*(\ml^k)-\mx^*(\ml^{k-1})} \\
      \leq& \frac{1}{\mu}\norm{\nabla f(\mx^*(\ml^k)) - \nabla f(\mx^*(\ml^{k-1}))} \\
      =& \frac{1}{\mu}\norm{\mA\T(\ml^k-\ml\p)} \\
      \leq& \frac{\os(\mA)}{\mu}\norm{\ml^k-\ml\p},
    }
    where the equality is obtained by applying the first-order optimality condition to $\mx^*(\ml^k)$ and $\mx^*(\ml^{k-1})$.
    It follows that
    \eqe{ \label{1146}
      &\norm{\mx\+-\mx^k} \\
      \leq& \norm{\mx\+-\mx^*(\ml^k)} + \norm{\mx^k-\mx^*(\ml^{k-1})} \\
      &+ \norm{\mx^*(\ml^k)-\mx^*(\ml^{k-1})} \\
      \leq& \frac{\os(\mA)}{\mu}\norm{\ml^k-\ml\p} + \delta\+ + \delta^k,
    }
    where the last inequality holds due to \cref{inner_error}.

    According to \cref{id2gt_z,average_z}, we have
    \eqe{ \label{ez}
      &\norm{\mz\+-\1\az\+} \\
      =& \norm{\mW\mz^k-\1\az^k + \lt(\mA-\1 \frac{1}{n}A\rt)(\mx\+-\mx^k)} \\
      \leq& \sigma\norm{\mz^k-\1\az^k} + \os(\mA)\norm{\mx\+-\mx^k} \\
      \leq& \sigma\norm{\mz^k-\1\az^k} + \frac{\os^2(\mA)}{\mu}\norm{\ml^k-\ml\p} \\
      &+ \os(\mA)\lt(\delta\+ + \delta^k\rt),
    }
    where the first inequality holds due to \cref{average,consensus} and the last one follows from \cref{1146}.

    According to \cref{id2gt_l}, \cref{consensus}, and \cref{ez}, we have
    \eqe{ \label{el}
      &\norm{\ml\+-\1\al\+} \\
      =& \norm{\mW\ml^k-\1\al^k + \beta(\mW\mz\+-\1\az\+)} \\
      \leq& \sigma\norm{\ml^k-\1\al^k} + \beta\sigma\norm{\mz\+-\1\az\+} \\
      \leq& \sigma\norm{\ml^k-\1\al^k} + \beta\sigma^2\norm{\mz^k-\1\az^k} \\
      &+ \frac{\beta\sigma\os^2(\mA)}{\mu}\norm{\ml^k-\ml\p}+ \beta\sigma\os(\mA)\lt(\delta\+ + \delta^k\rt).
    }

    Note that $\ax\+ = \mx^*\lt(\1\al^k\rt)$, similar to \cref{1734}, we have
    \eqe{
      \norm{\mx^*(\ml^k)-\ax\+} \leq \frac{\os(\mA)}{\mu}\norm{\ml^k-\1\al^k},
      \nonumber
    }
    it follows that
    \eqe{ \label{2241}
      &\norm{\az\+ - \frac{1}{n}(A\ax\+-b)} \\
      =& \frac{1}{n}\norm{A\lt(\mx\+-\mx^*(\ml^k) + \mx^*(\ml^k)-\ax\+\rt)} \\
      \leq& \frac{\os(A)\os(\mA)}{n\mu}\norm{\ml^k-\1\al^k} + \frac{\os(A)}{n}\delta\+,
    }
    where the equality holds due to \cref{average_z}.
    Let $\mx^*$ be the solution of \cref{mP}, it holds that
    \eqe{
      \mx^* &= \arg\min_{\mx \in \mR^{d}}\lt\{f(\mx) + \lambda^{*\T}(A\mx - b)\rt\} \\
      &= \mx^*(\1\lambda^*),
      \nonumber
    }
    then we have
    \eqe{ \label{2242}
      \norm{\ax\+-\mx^*} \leq&  \frac{1}{\mu}\norm{\mA\T\1\lt(\al^k-\lambda^*\rt)} \\
      \leq& \frac{\os(A)}{\mu}\norm{\al^k-\lambda^*}.
    }
    With \cref{{2241,2242}}, we can obtain that
    \eqe{
      &\norm{\mz\+} \\
      =& \bigg\|\mz\+-\1\az\+ + \1\az\+ - \1\frac{1}{n}(A\ax\+-b) \\
      &+ \1\frac{1}{n}A(\ax\+-\mx^*)\bigg\| \\
      \leq& \norm{\mz\+-\1\az\+} + \sqrt{n}\norm{\az\+ - \frac{1}{n}(A\ax\+-b)} \\
      &+ \frac{\os(A)}{\sqrt{n}}\norm{\ax\+-\mx^*} \\
      \leq& \sigma\norm{\mz^k-\1\az^k} + \frac{\os^2(\mA)}{\mu}\norm{\ml^k-\ml\p} \\
      &+ \frac{\os(A)\os(\mA)}{\sqrt{n}\mu}\norm{\ml^k-\1\al^k} + \frac{\os^2(A)}{\sqrt{n}\mu}\norm{\al^k-\lambda^*} \\
      &+ \lt(\os(\mA)+\frac{\os(A)}{\sqrt{n}}\rt)\delta\+ + \os(\mA)\delta^k,
    }
    it follows that
    \eqe{ \label{en}
      &\norm{\ml\+-\ml^k} \\
      =& \norm{(\mW-\mI)(\ml^k-\1\al^k) + \beta\mW\mz\+} \\
      \leq& (1+\sigma)\norm{\ml^k-\1\al^k} + \beta\norm{\mz\+} \\
      \leq& \lt(1+\sigma+\frac{\beta\os(A)\os(\mA)}{\sqrt{n}\mu}\rt)\norm{\ml^k-\1\al^k} \\
      &+ \beta\sigma\norm{\mz^k-\1\az^k} + \frac{\beta\os^2(\mA)}{\mu}\norm{\ml^k-\ml\p} \\
      &+ \frac{\beta\os^2(A)}{\sqrt{n}\mu}\norm{\al^k-\lambda^*} + \beta\os(\mA)\delta^k \\
      &+ \beta\lt(\os(\mA)+\frac{\os(A)}{\sqrt{n}}\rt)\delta\+.
    }

    Recall that $\ml^0 = 0$ (which implies that $\al^0 = 0$) and $0 < \beta < \frac{2n\mu}{\os^2(A)}$, also note that
    \eqe{
      \ax\+ &= \arg\min_{\mx \in \mR^{d}}\lt\{f(\mx) + \lt(\1\al^k\rt)\T\lt(\mA\mx-\mb\rt)\rt\} \\
      &= \arg\min_{\mx \in \mR^{d}}\lt\{f(\mx) + \al^{k\T}(A\mx - b)\rt\},
      \nonumber
    }
    then applying \cref{average} and \cref{contraction} gives that
    \eqe{ \label{2310}
      \norm{\al^k + \frac{\beta}{n}(A\ax\+-b) - \lambda^*_c} \leq \nu\norm{\al^k - \lambda^*_c}, \ \forall k \geq 0,
    }
    where $\nu = \max\lt\{\lt|1-\frac{\beta\os^2(A)}{n\mu}\rt|, \lt|1-\frac{\beta\us^2(A)}{nl}\rt|\rt\} \in (0, 1)$. With \cref{2310,average_z}, we can obtain that
    \eqe{ \label{eo}
      &\norm{\al\+-\lambda^*} \\
      \leq& \norm{\al^k+\frac{\beta}{n}(A\ax\+-b) -\lambda^*} + \frac{\beta}{n}\norm{A(\mx^*(\ml^k)-\ax\+)} \\
      &+ \frac{\beta}{n}\norm{A(\mx\+-\mx^*(\ml^k))} \\
      \leq& \nu\norm{\al^k-\lambda^*} + \frac{\beta\os(A)\os(\mA)}{n\mu}\norm{\ml^k-\1\al^k} + \frac{\beta\os(A)}{n}\delta\+.
    }

    Recall the definitions of $M$ and $H$, according to \cref{ez,el,en,eo}, we have
    \eqe{ \label{2307}
      \zeta\+ \leq M\zeta^k + H\xi^k,
    }
    combining \cref{2307} with $k=0, \cdots, k-1$ completes the proof.
  \end{appendix_proof}

  \begin{appendix_proof}[\cref{inner_iteration}]
    At $k$-th outer iteration, $g(x) = f_i(x) + \lambda_i^{k\T}\lt(A_ix-\frac{1}{n}b\rt)$, which is $\mu_i$-strongly convex and $l_i$-smooth.
    According to \cite[Section 2.2.1]{nesterov2018lectures}, we have
    \eqe{
    g(x^k) - g^* \leq \frac{l_i+\mu_i}{2}\norm{x_i^{k+1, 0}-x^*}^2e^{-k\sqrt{\frac{\mu_i}{l_i}}},
    \nonumber
    }
    it follows that
    \eqe{
    \norm{x^k-x^*}^2 \leq \frac{l_i+\mu_i}{\mu^3}\norm{\nabla f_i\lt(x_i^{k+1, 0}\rt)}^2e^{-k\sqrt{\frac{\mu_i}{l_i}}},
    \nonumber
    }
    then \cref{inner_error_i} can be guaranteed by letting
    \eqe{
    \frac{l_i+\mu_i}{\mu^3}\norm{\nabla f_i\lt(x_i^{k+1, 0}\rt)}^2e^{-k\sqrt{\frac{\mu_i}{l_i}}} \leq \frac{(\delta\+)^2}{n},
    \nonumber
    }
    which completes the proof.
  \end{appendix_proof}
\end{appendix}
\end{document}